\documentclass[12pt]{article}
\usepackage{makeidx}
\usepackage{amsfonts}
\usepackage{amsmath}

\newtheorem{theorem}{Theorem}[section]

\newtheorem{corollary}[theorem]{Corollary}

\newtheorem{definition}{Definition}[section]

\newtheorem{lemma}[theorem]{Lemma}
\newtheorem{notation}{Notation}

\newtheorem{proposition}[theorem]{Proposition}
\newtheorem{remark}{Remark}[section]

\newenvironment{proof}[1][Proof]{\textbf{#1.} }{\ \rule{0.5em}{0.5em}}

\begin{document}

\title{The Einstein relation for random walks \ on graphs}
\author{Andr\'{a}s Telcs \\
%EndAName
{\small Department of Computer Science and Information Theory, }\\
{\small University of Technology and Economcs Budapest}\\
{\small Goldmann Gy\"{o}rgy t\'{e}r 3, V2. 138}\\
{\small Budapest,}\\
{\small H-1111, HUNGARY}\\
{\small telcs@szit.bme.hu}}
\maketitle

\begin{abstract}
This paper investigates the Einstein relation; the connection between the
volume growth, the resistance growth and the expected time a random walk
needs to leave a ball on a weighted graph. The Einstein relation is proved
under different set of conditions. In the simplest case it is shown under
the volume doubling and time comparison principles. \ This and the other set
of conditions provide the basic vwork for the study of (sub-) diffusive
behavior of the random walks on weighted graphs.
\end{abstract}

\section{Introduction}

\setcounter{equation}{0}The study of diffusion dates back to Brown and
Einstein. In one of his celebrated works \cite{E} Einstein gave an explicit
formula for the expected value of the distance traveled by a particle in a
fluid,%
\begin{equation*}
\mathbb{E}\left[ d\left( X_{0},X_{t}\right) \right] =\sqrt{Dt},
\end{equation*}%
(where $d\left( x,y\right) $ stands for the distance) and for the diffusion
constant%
\begin{equation*}
D=\frac{k_{B}T}{6\pi \eta a},
\end{equation*}%
where $\eta $ is the viscosity of the fluid and $a$ is the radius of the
(assumed spherical) particle. For further historical remarks and explanation
see Hughes \cite{Hu}.

On typical fractals (cf.\cite{B}) one finds
\begin{equation*}
\mathbb{E}\left[ d\left( X_{0},X_{t}\right) \right] \simeq t^{\frac{1}{\beta
}}
\end{equation*}%
with an exponent $\beta \geq 2.$ Equivalently one can consider $E\left(
x,R\right) $ the mean of the exit time $T_{B\left( x,R\right) }$ needed by
the particle to leave the ball $B=B\left( x,R\right) $ centered on $x$ of
radius $R$. \ For many fractals ( cf. \cite{BWv},\cite{HK}) this quantity
grows polynomially with $\beta >0:$%
\begin{equation*}
E\left( x,R\right) =\mathbb{E}\left( T_{B}|X_{0}=x\right) \simeq R^{\beta }.
\end{equation*}

This is the reason why the relation%
\begin{equation}
\beta =\alpha -\gamma  \label{Erdim}
\end{equation}%
is called the Einstein relation by Alexander and Orbach \cite{AO}. \ In $%
\left( \ref{Erdim}\right) $ the exponent $\beta $ is the diffusion exponent
(or walk dimension), $\alpha $ is the the fractal dimension, governing the
volume growth, and $\gamma $ is the conductivity (or capacity) exponent
(exponent of the conductivity between of the surfaces of the annuli).

In the last two decades the sub-diffusive behavior of fractal spaces (see
\cite{B},\cite{HK} as starting references) was intensively studied.
Two-sided heat kernel estimates have been proved for particular fractals and
for wide class of spaces and graphs as well. In almost all the cases the
mean exit time has been found to satisfy%
\begin{equation*}
E\left( x,R\right) \simeq R^{\beta }
\end{equation*}%
and the Einstein relation is still in the heart of the matter. Here and in
what follows $a_{\xi }\simeq b_{\xi }$ means that there is a $C>1$ such that
$C^{-1}a_{\xi }\leq b_{\xi }\leq Ca_{\xi }$ for all $\xi .$

The major challenge in the study of diffusion is to find connection between
geometric, analytic, spectral and other properties of the space and behavior
of diffusion.

In order to formulate the main topics of the present paper let us switch to
the discrete space and time situation, to random walks, which are known as
excellent models for diffusion. \ They exhibit almost all the interesting
phenomena and the theoretical difficulties, but some technical problems can
be avoided by their usage. It is well-known that for the simple \ symmetric
nearest neighbor random walk $X_{n}$ on $%
%TCIMACRO{\U{2124} }%
%BeginExpansion
\mathbb{Z}
%EndExpansion
^{d}$ the expected value of the traversed distance at time $n$ is
\begin{equation}
\mathbb{E}\left( d\left( X_{0},X_{n}\right) \right) =c_{d}\sqrt{n},
\label{dmpl}
\end{equation}%
where $d\left( x,y\right) $ is the shortest path graph distance in $x,y\in
%TCIMACRO{\U{2124} }%
%BeginExpansion
\mathbb{Z}
%EndExpansion
^{d}$. It is also well-known that the mean exit time%
\begin{equation}
E\left( x,R\right) =\mathbb{E}\left( T_{B}|X_{0}=x\right) =C_{d}R^{2}
\label{exitt}
\end{equation}%
in perfect agreement with $\left( \ref{dmpl}\right) .$

It has been previously shown by the author \cite{T1} that $\left( \ref{Erdim}%
\right) $ holds for a large class of graphs. \ A more detailed picture can
be obtained by considering the resistance and volume growth properties. \
Let $V\left( x,R\right) $ denote the volume of the ball $B\left( x,R\right) $%
. Let $\rho \left( x,r,R\right) $ denote the resistance of an annulus $%
B\left( x,R\right) \backslash B\left( x,r\right) $, i.e. the resistance
between the inner and outer surface and let $v\left( x,r,R\right) $ denote
the volume of the annuli:
\begin{equation*}
v=v(x,r,R)=V(x,R)-V(x,r).
\end{equation*}%
Recent studies ( cf. \ \cite{B},\cite{BB},\cite{GT1},\cite{GT2},\cite{TD})
show that the relevant form of the Einstein relation $\left( \mathbf{ER}%
\right) $ is
\begin{equation}
E\left( x,2R\right) \simeq \rho \left( x,R,2R\right) v\left( x,R,2R\right) .
\label{ER}
\end{equation}%
Our aim in the present paper is to give reasonable conditions for this
relation and show some further properties of the mean exit time which are
essential in the investigation of diffusion, in particularly to obtain heat
kernel estimates. All the theorems presented here are new, the
multiplicative form of the Einstein relation obtained is a significant
improvement over $\left( \ref{Erdim}\right) $ (cf. \cite{T1}). Similar
estimates for particular structures or under stronger conditions have been
considered (cf. \cite{BWv},\cite{B}\cite{BB},\cite{HK} and under stronger
conditions by the author in \cite{T1},\cite{GT1},\cite{GT2},\cite{Tlocal},%
\cite{TD}) but to the author's best knowledge ther are no comparable results
in this generality. The imposed conditions seems to be strong and hard to
check but recent studies (\cite{GT1},\cite{GT2},\cite{Tfull},\cite{Tweak})
show that the conditions not only sufficient but necessary for upper- and
two-sided heat kernel estimates. The Einstein relation provides a simple
connection between the members of the triplet of mass, resistance and mean
exit time. Examples show that without some natural restrictions any two of
them are "independent" (cf. Lemma 5.1, 5.2 \cite{BWv} and \cite{B} and
references there). From physical point of view it seems to natural to impose
conditions on mass and resistance.

It is interesting that the conditions which proved to be most natural for
the Einstein relation are also those which provide conditions for the much
deeper study of the heat kernel. We hope that beyond the actual results the
paper leads to a better understanding of diffusion.

The structure of the paper is the following. \ Basic definitions are
collected in Section \ref{sbasic}. In the consecutive sections we gradually
change the set of conditions. The changes has two\ aspects. We start with a
condition on the mean exit time which might be challenged as input in the
study the diffusion. In order to eliminate this deficiency we replace this
condition with a pair of conditions in Section \ref{sER2} which reflect
resistance properties. On the other hand the conditions will become more and
more restrictive to meet the needs of the heat kernel estimates. In
particular the strong assumption of the elliptic Harnack inequality is used.

Section \ref{sineq} contains general inequalities and the first theorem on
the Einstein relation which is based mainly on regularity conditions imposed
on the volume growth and mean exit time. \ In Section \ref{shar} a key
observation is made on the growth of the resistance if the elliptic Harnack
inequality is satisfied. \ Section \ref{sER1},\ref{sER2} \ and \ref{ssad}
provide two more result on the Einstein relation under different conditions
and several further properties of the mean exit time are discussed. Section %
\ref{shar} and \ref{sER2} contains several remarks and observations which
provide an insight on the interplay of the used conditions and hopefully
also leads to the better understanding of the nature of the elliptic Harnack
inequality.

\subsection{Acknowledgement}

The author expresses his sincere thanks to the referees for the elaborated
remarks and proposals how to improve the paper. \ Special thanks to Prof.
Alexander Grigor'yan, who suggested to summarize and complete the scattered
results around the Einstein relation in one paper.

\section{Basic definitions}

\setcounter{equation}{0}\label{sbasic}Let us consider a countable infinite
connected graph $\Gamma $. \ A weight function $\mu _{x,y}=\mu _{y,x}>0$ is
given on the edges $x\sim y.$ This weight induces a measure%
\begin{equation*}
\mu (x)=\sum_{y\sim x}\mu _{x,y},\text{ }\mu (A)=\sum_{y\in A}\mu (y)
\end{equation*}%
on the vertex set $A\subset \Gamma $ and defines a reversible Markov chain $%
X_{n}\in \Gamma $, i.e. a random walk on the weighted graph $(\Gamma ,\mu )$
with transition probabilities
\begin{align*}
P(x,y)& =\frac{\mu _{x,y}}{\mu (x)}, \\
P_{n}(x,y)& =\mathbb{P}(X_{n}=y|X_{0}=x).
\end{align*}

\begin{definition}
The weighted graph is equipped with the inner product: for $f,g\in
c_{0}\left( \Gamma \right) $ (set of finitely supported functions over $%
\Gamma $)%
\begin{equation*}
\left( f,g\right) =\left( f,g\right) _{\mu }=\sum_{x}f\left( x\right)
g\left( x\right) \mu \left( x\right)
\end{equation*}
\end{definition}

The graph is equipped with the usual (shortest path length) graph distance $%
d(x,y).$ Open metric balls centered on $x\in \Gamma ,$ of radius $R>0$ are
defined as
\begin{equation*}
B(x,R)=\{y\in \Gamma :d(x,y)<R\},
\end{equation*}%
the surface by%
\begin{equation*}
S(x,R)=\{y\in \Gamma :d(x,y)=R\}
\end{equation*}%
and the $\mu -$measure of an open ball is denoted by
\begin{equation}
V(x,R)=\mu \left( B\left( x,R\right) \right) .  \label{vdef}
\end{equation}

\begin{definition}
The weighted graph has the \emph{volume doubling }$\left( \mathbf{VD}\right)
$ property if there is a constant $D_{V}>0$ such that for all $x\in \Gamma $
and $R>0$%
\begin{equation}
V(x,2R)\leq D_{V}V(x,R).  \label{PD1V}
\end{equation}
\end{definition}

\begin{definition}
\label{BC}The bounded covering condition $\left( \mathbf{BC}\right) $ holds
if there is an integer $K$ such that for all $x\in \Gamma ,R>0$ the ball $%
B\left( x,2R\right) $ \ can be covered by at most $K$ balls of radius $R.$
\end{definition}

\begin{remark}
It is well-known that volume doubling property implies the bounded covering
condition on graphs. (cf. Lemma 2.7 of \cite{BB}.)
\end{remark}

\begin{notation}
For a set $A\subset \Gamma $ denote the closure by
\begin{equation*}
\overline{A}=\left\{ y\in \Gamma :\text{there is an }x\in A\text{ such that }%
x\sim y\right\} \text{.}
\end{equation*}%
The external boundary is defined as $\partial A=\overline{A}\backslash A.$%
\bigskip
\end{notation}

\begin{definition}
We say that condition $\mathbf{(p}_{0}\mathbf{)}$ holds if there is a
universal $p_{0}>0$ such that for all $x,y\in \Gamma ,x\sim y$
\begin{equation}
\frac{\mu _{x,y}}{\mu (x)}\geq p_{0}.  \label{p0}
\end{equation}
\end{definition}

The next proposition is taken from \cite{GT1} (see also \cite{Tfull})

\begin{proposition}
\label{Pregvol}\label{plocvol}If $\left( p_{0}\right) $ holds, then, for all
$x,y\in \Gamma $ and $R>0$ and for some $C>1$,
\begin{equation}
V(x,R)\leq C^{R}\mu (x),  \label{vbound}
\end{equation}%
\begin{equation}
p_{0}^{d(x,y)}\mu (y)\leq \mu (x)
\end{equation}%
and for any $x\in \Gamma $
\begin{equation}
\left\vert \left\{ y:y\sim x\right\} \right\vert \leq \frac{1}{p_{0}}.
\label{deg}
\end{equation}
\end{proposition}

\begin{remark}
It is easy to show (cf. \cite{CG}) that the volume doubling property implies
an \textbf{anti-doubling property}: there is a constant $A_{V}>1$ such that
for all $x\in \Gamma ,R>0$%
\begin{equation}
2V(x,R)\leq V(x,A_{V}R).  \label{aVD}
\end{equation}%
One can also show that $\left( VD\right) $ is equivalent to
\begin{equation}
\frac{V(x,R)}{V(y,S)}\leq C\left( \frac{R}{S}\right) ^{\alpha },  \label{VC}
\end{equation}%
where $\alpha =\log _{2}D_{V}$,$d(x,y)\leq R$ and the anti-doubling property
$\left( \ref{aVD}\right) $ is equivalent to the existence of $c,\alpha
_{1}>0 $ such that for all $x\in \Gamma ,R>S>0$%
\begin{equation}
\frac{V(x,R)}{V(x,S)}\geq c\left( \frac{R}{S}\right) ^{\alpha _{1}}.
\label{adV}
\end{equation}
\end{remark}

\begin{definition}
We say that the weak volume comparison condition $\left( \mathbf{wVC}\right)
$ holds if here is a $C>1$ such that for all $x\in \Gamma ,R>0,y\in B\left(
x,R\right) $%
\begin{equation}
\frac{V(x,R)}{V(y,R)}\leq C.  \label{wVC}
\end{equation}
\end{definition}

\begin{remark}
One can also easily verify that
\begin{equation*}
\left( wVC\right) +\left( BC\right) \Longleftrightarrow \left( VD\right) .
\end{equation*}
\end{remark}

\subsection{The mean exit time}

Let us introduce the exit time $T_{A}.$

\begin{definition}
The exit time from a set $A$ is defined as
\begin{equation*}
T_{A}=\min \{k:X_{k}\in \Gamma \backslash A\},
\end{equation*}%
its expected value is denoted by
\begin{equation*}
E_{z}(A)=\mathbb{E}(T_{A}|X_{0}=z),
\end{equation*}%
and let us use the short notation
\begin{equation*}
E_{z}(x,R)=\mathbb{E}(B(x,R)|X_{0}=z)
\end{equation*}%
and $A=$ $B(x,R)$
\begin{equation*}
E(x,R)=E_{x}(x,R).
\end{equation*}
\end{definition}

\begin{definition}
We will say that the weighted graph $(\Gamma ,\mu )$ satisfies the time
comparison principle $\left( \mathbf{TC}\right) $ if there is a constant $%
C_{T}>1$ such that for all $x\in \Gamma $ and $R>0,y\in B\left( x,R\right) $%
\begin{equation}
\frac{E(x,2R)}{E\left( y,R\right) }\leq C_{T}.  \label{TC}
\end{equation}
\end{definition}

\begin{definition}
We will say that $(\Gamma ,\mu )$ has \emph{time doubling} property $\left(
\mathbf{TD}\right) $ if there is a $D_{T}>0$ such that for all $x\in \Gamma $
and $R\geq 0$%
\begin{equation}
E(x,2R)\leq D_{T}E(x,R).  \label{TD}
\end{equation}
\end{definition}

\begin{remark}
It is clear that $\left( TC\right) $ implies $\left( TD\right) $ setting $%
y=x.$
\end{remark}

\begin{remark}
The time comparison principle evidently implies the following weaker form of
the time comparison principle $\left( \mathbf{wTC}\right) $: there is a $C>0$
such that
\begin{equation}
\frac{E(x,R)}{E(y,R)}\leq C  \label{pd2e}
\end{equation}%
\ for all $x\in \Gamma ,R>0,y\in B\left( x,R\right) .$ One can observe that $%
\left( \ref{pd2e}\right) $ is the difference between $\left( TC\right) $ and
$\left( TD\right) .$ It is easy to see that
\begin{equation*}
\left( TC\right) \Longleftrightarrow \left( TD\right) +\left( wTC\right) .
\end{equation*}
\end{remark}

\begin{remark}
From $\left( TD\right) $ it follows that there are $C>0$ and $\beta >0$ such
that for all $x\in \Gamma $ \ and $R>S>0$%
\begin{equation}
\frac{E(x,R)}{E(x,S)}\leq C\left( \frac{R}{S}\right) ^{\beta }
\label{tdbeta}
\end{equation}%
and $\left( TC\right) $ is equivalent to%
\begin{equation}
\frac{E(x,R)}{E(y,S)}\leq C\left( \frac{R}{S}\right) ^{\beta }
\label{tcbeta}
\end{equation}%
for any $y\in B\left( x,R\right) .$ One can take that $\beta =\log
_{2}C_{T}. $ Later (cf. Corollary \ref{bb>1} and \ref{ce>2}) we shall see
that $\beta \geq 1$ in general and $\beta \geq 2$ under some natural
conditions.
\end{remark}

\begin{definition}
The maximal mean exit time is defined as
\begin{equation*}
\overline{E}(A)=\max_{x\in A}E_{x}(A),
\end{equation*}%
in particular the notation $\overline{E}(x,R)=\overline{E}(B(x,R))$ will be
used.
\end{definition}

\begin{definition}
We introduce the condition $(\overline{\mathbf{E}})$ which means that there
is a constant $C>0$ such that for all $x\in \Gamma ,R>0$
\begin{equation}
\overline{E}(x,R)\leq CE(x,R).  \label{Ebar}
\end{equation}
\end{definition}

\begin{remark}
One can see easily that
\begin{equation*}
\left( TC\right) \Longrightarrow (\overline{E}).
\end{equation*}
\end{remark}

\subsection{The Laplace operator}

\begin{definition}
The random walk on the weighted graph is a reversible Markov chain with
respect to $\mu \left( x\right) $ and the Markov operator $P$ \ is naturally
defined by%
\begin{equation*}
Pf\left( x\right) =\sum P\left( x,y\right) f\left( y\right) .
\end{equation*}
\end{definition}

\begin{definition}
The Laplace operator on the weighted graph $\left( \Gamma ,\mu \right) $ is
defined simply as
\begin{equation*}
\Delta =P-I.
\end{equation*}
\end{definition}

\begin{definition}
For $A\subset \Gamma $ consider $P^{A}$, the restriction of the Markov
operator $P$ to $A.$ This operator is the Markov operator of the killed
Markov chain, which is killed on leaving $A$.$\ $Its iterates are denoted by
$P_{k}^{A}$.
\end{definition}

\begin{definition}
The Laplace operator with Dirichlet boundary conditions on a finite set $%
A\subset \Gamma $ is defined as%
\begin{equation*}
\Delta ^{A}f\left( x\right) =\left\{
\begin{array}{ccc}
\Delta f\left( x\right) & \text{if} & x\in A \\
0 & if & x\notin A%
\end{array}%
\right. .
\end{equation*}%
The smallest eigenvalue of $-\Delta ^{A}$ is denoted in general by $\lambda
(A)$ and for $A=B(x,R)$ it is denoted by $\lambda (x,R)=\lambda (B(x,R)).$
\end{definition}

\begin{definition}
We introduce
\begin{equation*}
G^{A}(y,z)=\sum_{k=0}^{\infty }P_{k}^{A}(y,z)
\end{equation*}%
the local Green function, the Green function of the killed walk and the
corresponding Green kernel as
\begin{equation*}
g^{A}(y,z)=\frac{1}{\mu \left( z\right) }G^{A}(y,z).
\end{equation*}
\end{definition}

\begin{remark}
One can observe that the local Green function $G^{A}\left( x,y\right) $ is
nothing else than the expected number of visits of the site $y$ by the
killed walk starting on $x.$
\end{remark}

\subsection{The resistance}

\begin{definition}
For any two disjoint sets, $A,B\subset \Gamma ,$ the resistance, $\rho
(A,B), $ is defined as
\begin{equation*}
\rho (A,B)=\left( \inf \left\{ \left( \Delta f,f\right) _{\mu
}:f|_{A}=1,f|_{B}=0\right\} \right) ^{-1}
\end{equation*}%
and we introduce
\begin{equation*}
\rho (x,r,R)=\rho (B(x,r),\Gamma \backslash B(x,R))
\end{equation*}%
for the resistance of the annulus around $x\in \Gamma ,$ with $R>r>0$.
\end{definition}

This formal definition is in full agreement with the natural physical
interpretation. \ If we consider the edges of the graph as wires of
conductance $w_{x,y}$ the graph forms an electric network. Then $\rho (A,B)$
is the resistance ( $1/\rho (A,B)$ the conductance ) which can be measured
in the electric network if the two poles of a power source are connected to
the sets $A$ and $B.$

\section{Basic inequalities}

\setcounter{equation}{0}$\label{sineq}$ This section collects several known
and some new inequalities which connect volume, mean exit time, resistance
and smallest eigenvalue of the Laplacian of finite sets. We mainly work
under the set of condition $\left( p_{0}\right) ,\left( VD\right) $ and $%
\left( TC\right) .$ Alone these conditions are not enough strong to obtain
on- and off-diagonal heat kernel upper and off-diagonal lower bounds but
imply the Einstein relation as the following theorem shows and they are
essential in the study of the heat kernel (cf. \cite{Tfull},\cite{Tweak}).

\begin{theorem}
\label{tallcc}$\left( p_{0}\right) ,\left( VD\right) $ and $\left( TC\right)
$ \ implies
\begin{equation*}
\lambda ^{-1}\left( x,2R\right) \asymp E\left( x,2R\right) \asymp \overline{E%
}\left( x,2R\right) \asymp \rho \left( x,R,2R\right) v\left( x,R,2R\right) .
\end{equation*}
\end{theorem}

The proof is given via a series of statements. \

\begin{lemma}
\label{llrv} For all weighted graphs $\left( \Gamma ,\mu \right) $ and for
all finite sets $A\subset B\subset \Gamma $ the inequality
\begin{equation}
\lambda (B)\rho (A,\Gamma \backslash B)\mu (A)\leq 1,  \label{lrm}
\end{equation}%
holds, particularly%
\begin{equation}
\lambda (x,2R)\rho (x,R,2R)V(x,R)\leq 1.  \label{lrvb}
\end{equation}
\end{lemma}

\begin{proof}
The reader is requested to consult Lemma 4.6 \cite{Tlocal}.
\end{proof}

\begin{lemma}
\label{lebar}For any finite set $A\subset \Gamma $%
\begin{equation}
\lambda ^{-1}\left( A\right) \leq \overline{E}\left( A\right)  \label{llebar}
\end{equation}
\end{lemma}

\begin{proof}
Please see Lemma 3.6 \cite{TD}.
\end{proof}

\begin{lemma}
\label{lmarkov}On all $\left( \Gamma ,\mu \right) $ for $\ $any $x\in \Gamma
,R>r>0$%
\begin{equation*}
E\left( x,R+r\right) \geq E\left( x,R\right) +\min_{y\in r\left( x,R\right)
}E\left( y,r\right) .
\end{equation*}
\end{lemma}

\begin{proof}
First let us observe that from the triangle inequality it follows that for
any $y\in S\left( x,R\right) $
\begin{equation*}
B\left( y,r\right) \subset B\left( x,R+r\right) .
\end{equation*}%
From this and from the strong Markov property one obtains that%
\begin{eqnarray*}
E\left( x,R+r\right) &=&\mathbb{E}_{x}\left( T_{B}+E_{X_{T_{B}}}\left(
x,R+r\right) \right) \\
&\geq &E\left( x,R\right) +\mathbb{E}_{x}\left( E_{X_{T_{B}}}\left(
X_{T_{B}},S\right) \right) .
\end{eqnarray*}%
But $X_{T_{B}}\in S\left( x,R\right) $ which gives the statement.
\end{proof}

\begin{corollary}
\label{bb>1}The mean exit time $E\left( x,R\right) $ for $R\in \mathbb{N}$
is strictly monotone and has inverse $e\left( x,n\right) :\Gamma \times
%TCIMACRO{\U{2115} }%
%BeginExpansion
\mathbb{N}
%EndExpansion
\rightarrow
%TCIMACRO{\U{2115} }%
%BeginExpansion
\mathbb{N}
%EndExpansion
$%
\begin{equation*}
e\left( x,n\right) =\min \left\{ R\in
%TCIMACRO{\U{2115} }%
%BeginExpansion
\mathbb{N}
%EndExpansion
:E\left( x,R\right) \geq n\right\} .
\end{equation*}
\end{corollary}

\begin{proof}
Simply let $S=1$ in Lemma \ref{lmarkov} and use that $E\left( x,1\right)
\geq 1.$
\end{proof}

\begin{lemma}
\label{lE<rm}On all $\left( \Gamma ,\mu \right) $ for $\ $any $x\in A\subset
\Gamma $%
\begin{equation*}
E_{x}\left( T_{A}\right) \leq \rho \left( \left\{ x\right\} ,\Gamma
\backslash A\right) \mu \left( A\right) .
\end{equation*}
\end{lemma}

\begin{proof}
This observation is well-known ( cf. \cite{T1} equation $\left( 1.5\right) $%
, or \cite{BWv}) therefore we give the proof in a concise form. \ Denote $%
\tau _{y}$ the first hitting time of $y\in A$ and $F^{A}y,=\mathbb{P}%
_{y}\left( \tau _{x}<T_{A}\right) \leq 1$ the hitting probability. Then
\begin{eqnarray*}
g^{A}\left( x,y\right) &=&g^{A}\left( y,x\right) =F^{A}\left( y,x\right)
g^{A}\left( x,x\right) \\
&\leq &g^{A}\left( x,x\right) =\rho \left( \left\{ x\right\} ,\Gamma
\backslash A\right) ,
\end{eqnarray*}%
where the last equality follows from the interpretation of the capacity
potential (cf. \cite{BWv}). The mean exit time can be decomposed and
estimated as follows%
\begin{eqnarray*}
E_{x}\left( T_{A}\right) &=&\sum_{y\in A}G^{A}\left( x,y\right) =\sum_{y\in
A}g^{A}\left( x,y\right) \mu \left( y\right) \\
&\leq &g^{A}\left( x,x\right) \sum_{y\in A}\mu \left( y\right) =\rho \left(
\left\{ x\right\} ,\Gamma \backslash A\right) \mu \left( A\right) .
\end{eqnarray*}
\end{proof}

Let $A\subset \Gamma $. We define a new graph $\Gamma ^{a},$ the graph which
is obtained by shrinking the set $A$ \ into a single vertex $a.$ The graph $%
\Gamma ^{a}$ has the vertex set $\Gamma ^{a}=\Gamma \backslash A\cup \left\{
a\right\} ,$ where $a$ is a new vertex. The edge set contains all edges $%
x\sim y$\ for $x,y\in \Gamma \backslash A$ and their weights unaltered $\mu
_{x,y}^{a}=\mu _{x,y}.$ There is an edge between $x\in \Gamma \backslash A$
and $a$ if there is a vertex $y\in A$ for which $x\sim y$ and the weights
are defined by $\mu _{x,a}^{a}=\sum_{y\in A}\mu _{x,y}.$ The random walk on $%
\Gamma ^{a}$ is defined as in general on weighted graphs. $\,$

\begin{corollary}
\label{cE<rm}For $\left( \Gamma ,\mu \right) $ and for finite sets $A\subset
B\subset \Gamma $ consider $\left( \Gamma ^{a},\mu ^{a}\right) $ and the
corresponding random walk. Then%
\begin{equation}
E_{a}\left( T_{B}\right) \leq \rho \left( A,\Gamma \backslash B\right) \mu
\left( B\backslash A\right) .  \label{ta1}
\end{equation}
\end{corollary}

\begin{proof}
The statement is an immediate consequence of Lemma \ref{lE<rm}.
\end{proof}

\begin{lemma}
\label{lminE<rv}For $\left( \Gamma ,\mu \right) $ for all $x\in \Gamma ,R>0,$%
\begin{equation*}
\min_{z\in \partial B\left( x,\frac{3}{2}R\right) }E\left( z,R/2\right) \leq
\rho \left( x,R,2R\right) v\left( x,R,2R\right) .
\end{equation*}
\end{lemma}

\begin{proof}
Consider the annulus $D=B\left( x,2R\right) \backslash B\left( x,R\right) $.
Apply Corollary \ref{cE<rm} for $A=B\left( x,R\right) ,B=B\left( x,2R\right)
$ to obtain
\begin{equation}
\rho \left( x,R,2R\right) v\left( x,R,2R\right) \geq E_{a}\left(
T_{B}\right) .  \label{ta}
\end{equation}%
It is clear that the walk started in $a$ and leaving $B$ should cross $%
\partial B\left( x,3/2R\right) .$ \ Now we use the Markov property as in
Lemma \ref{lmarkov}. Denote the first hitting (random) point by $\xi .$
Again evident that the walk continued from $\xi $ should leave $B\left( \xi ,%
\frac{1}{2}R\right) $ before it leaves $B\left( x,2R\right) .$ \ This means
that
\begin{eqnarray*}
&&\rho \left( x,R,2R\right) v\left( x,R,2R\right) \\
&\geq &E_{a}\left( T_{B}\right) \geq \min_{y\in \partial B\left( x,\frac{3}{2%
}R\right) }E\left( y,\frac{1}{2}R\right) .
\end{eqnarray*}
\end{proof}

\begin{theorem}
\label{tLER}\label{tER1}If $\left( p_{0}\right) ,\left( VD\right) ,(TC)$
hold then
\begin{equation}
E(x,2R)\simeq \rho (x,R,2R)v(x,R,2R).  \label{LER}
\end{equation}
\end{theorem}

\begin{proof}[Proof of Theorem \protect\ref{tLER}]
Let us recall that
\begin{equation*}
\overline{E}(x,R)=\max_{z\in B\left( x,R\right) }E_{z}(x,R).
\end{equation*}%
We start with the general inequalities $\left( \ref{lrvb}\right) $ and $%
\left( \ref{llebar}\right) ;$
\begin{equation}
\rho (x,R,2R)V(x,R)\leq \lambda ^{-1}(x,2R)\leq \overline{E}(x,2R)\leq
CE(x,2R)  \label{llcce}
\end{equation}%
where in the last step $\left( TC\right) $ is used. For the upper estimate
let us apply Lemma \ref{lminE<rv}
\begin{equation*}
\rho (x,R,2R)V(x,2R)\geq \min_{y\in \partial B\left( x,\frac{3}{2}R\right)
}E(y,\frac{1}{2}R).
\end{equation*}%
Finally from $\left( TC\right) $ it follows that%
\begin{equation*}
\rho (x,R,2R)v\left( x,R,2R\right) \geq cE(x,2R).
\end{equation*}
\end{proof}

\begin{proof}[Proof of Theorem \protect\ref{tallcc}]
The combination of $\left( \ref{LER}\right) $ and $\left( \ref{llcce}\right)
$ gives the result.
\end{proof}

\begin{proposition}
\label{pPD3E}\label{padE}If $\left( wTC\right) $ holds then anti-doubling
property holds for $E\left( x,R\right) ,$ which means that there is a
constant $A>1$ such that for all $x\in \Gamma ,R>0$%
\begin{equation}
E\left( x,AR\right) >2E\left( x,R\right) .  \label{aDE}
\end{equation}
\end{proposition}

\begin{proof}
Consider any $y\in S(x,2R)$ and apply twice $\left( wTC\right) $ to get
\begin{equation*}
E(y,R)>cE(x,R).
\end{equation*}%
Now again from the Markov property for the stopping time $T_{B\left(
x,3R\right) }$ we obtain
\begin{equation*}
E(x,3R)>E(x,2R)+\min_{y\in S(x,2R)}E(y,R)\geq (1+c)E(x,R)
\end{equation*}%
and iterating this procedure a sufficient number of times we get the result.
\end{proof}

\begin{proposition}
\label{pra>l2} For all weighted graphs and for all finite sets with $%
A\subset B\subset \Gamma $%
\begin{equation*}
\rho (A,\partial B)\mu (B\backslash A)\geq d(A,\partial B)^{2}.
\end{equation*}
\end{proposition}

\begin{proof}
The proof follows the idea of Lemma 1 of \cite{T1}. Denote $L=d(A,\partial
B) $, and $S_{i}=\{z\in B:d(A,z)=i\},S_{0}=A,S_{L}=\partial B$ and $%
E_{i}=\{(x,y):x\in S_{i},y\in S_{i+1}\},\mu \left( E_{i}\right)
=\sum_{\left( z,w\right) \in E_{i}}\mu _{z,w}.$ Using these conventions one
obtains that%
\begin{equation*}
\rho (A,\partial B)\geq \sum_{i=0}^{L-1}\rho (S_{i},S_{i+1})=\sum_{i=0}^{L-1}%
\frac{1}{\mu (E_{i})}\geq \frac{L^{2}}{\sum_{i=0}^{L-1}\mu (E_{i})}.
\end{equation*}
\end{proof}

This proposition has an interesting consequence.

\begin{corollary}
\label{crv>r2}\label{crav>l2 copy(1)}For all weighted graphs, if $x\in
\Gamma ,R\geq r\geq 0,$ then%
\begin{equation}
\rho (x,r,R)v(x,r,R)\geq (R-r)^{2},  \label{rv>dis^2}
\end{equation}
\end{corollary}

\begin{proof}
The statement is immediate from Proposition \ref{pra>l2}.
\end{proof}

\begin{proposition}
\label{pll<2}If $\left( p_{0}\right) $ and $(VD)$ hold then there is a $c>0$
such that for all $x\in \Gamma ,R>0$%
\begin{equation*}
\lambda ^{-1}(x,R)\geq cR^{2}.
\end{equation*}
\end{proposition}

\begin{proof}
The inequality follows from $\left( \ref{lrvb}\right) ,\left( \ref{rv>dis^2}%
\right) $ and $\left( VD\right) .$
\end{proof}

\begin{corollary}
\label{ce>2}If $\left( p_{0}\right) $,$(VD)$ and $(\overline{E})$ hold then
there is a $c>0$ such that for all $x\in \Gamma ,R>0$%
\begin{equation}
E(x,R)\geq cR^{2}.  \label{E>r2}
\end{equation}
\end{corollary}

\begin{proof}
The statement follows easily from Lemma \ref{lebar} and \ref{pll<2}.
\end{proof}

\begin{remark}
If one has the upper bound
\begin{equation*}
E\left( x,R\right) \leq CR^{\beta }
\end{equation*}%
for the mean exit time with a given $\beta >0$ then $\left( \ref{E>r2}%
\right) $ immediately implies%
\begin{equation*}
\beta \geq 2.
\end{equation*}
\end{remark}

\begin{remark}
Since $\left( TC\right) \Longrightarrow (\overline{E})$ we also have the
implication $\left( p_{0}\right) ,(VD),(TC)\Longrightarrow \left( \ref{E>r2}%
\right) .$
\end{remark}

\section{The Harnack inequality and the Green kernel}

\setcounter{equation}{0}\label{shar} In this section we study the
relationship between the elliptic Harnack inequality (see Definition \ref{dH}
below) and resistance properties of the graph. The main results of this
section (and in some extent of the paper as well) are Theorem \ref{tannhar},
Corollary \ref{crrd} and \ref{csumm}. \ The section ends with several
further remarks which connect the constants and exponents popping up in the
resistance, mean exit time and volume estimates and contributes to the
related observations given in \cite{B} and \cite{LW}. Further studies in
this direction may disclose the nature of the constant in the elliptic
Harnack inequality the "Harnack constant".

\begin{definition}
A function $h$ is harmonic on a set $A\subset \Gamma $ if it is defined on $%
\overline{A}$ and
\begin{equation*}
Ph\left( x\right) =\sum_{y}P\left( x,y\right) h\left( y\right) =h\left(
x\right)
\end{equation*}%
for all $x\in A.$
\end{definition}

\begin{definition}
\label{dH}We say that the weighted graph $(\Gamma ,\mu )$ satisfies \emph{%
the elliptic Harnack inequality }$\left( \mathbf{H}\right) $ if, for all $%
x\in \Gamma ,R>0$ and for any non-negative harmonic function $u$ which is
harmonic in $B(x,2R)$, the following inequality holds
\begin{equation}
\max_{B(x,R)}u\leq H\min_{B(x,R)}u\,,  \label{H}
\end{equation}%
with some constant $H>1$ independent of $x$\ and $R.$
\end{definition}

\begin{remark}
\label{Remp0}One can check easily that for any fixed $R_{0}$ for all $%
R<R_{0} $ the Harnack inequality follows from $\left( p_{0}\right) $.
\end{remark}

\begin{definition}
\ We say that $\left( \Gamma ,\mu \right) $ satisfies $\left( \mathbf{HG}%
\right) $ \emph{the Harnack inequality for Green functions }if there is a $%
C>1,$ such that for all $x\in \Gamma $ and $R>0$ and for any finite set $%
U\supset B(x,2R)$,
\begin{equation}
\sup_{y\notin B(x,R)}g^{U}(x,y)\leq C\inf_{z\in B(x,R)}g^{U}(x,z).  \tag{HG}
\label{HG2}
\end{equation}
\end{definition}

For more concise treatment let us define two further inequalities which are
equivalent to $\left( HG\right) $. There is a $C>1,$ such that for all $x\in
\Gamma $ and $R>r>0,$ if $B=B\left( x,2R\right) $, then
\begin{equation}
\sup_{y\notin B(x,r)}g^{B}(x,y)\leq C\inf_{z\in B(x,r)}g^{B}(x,z).
\label{HGrR}
\end{equation}%
There is a $C>1,$ such that for all $x\in \Gamma $ and $R>0,$ if $B=B\left(
x,2R\right) $, \ then%
\begin{equation}
\sup_{y\notin B(x,R)}g^{B}(x,y)\leq C\inf_{z\in B(x,R)}g^{B}(x,z).
\label{HG2R}
\end{equation}

\begin{proposition}
\label{PG=>H}Assume that $(p_{0})$ holds on the graph $(\Gamma ,\mu )$. Then%
\begin{equation*}
(HG)\Longrightarrow (H).
\end{equation*}
\end{proposition}

The proof can be found in \cite{GT1}. \ The next two propositions are
Proposition 4.3 and 4.4 from \cite{GT2}.

\begin{proposition}
\label{PG=res}Assume that the graph $\left( \Gamma ,\mu \right) $ satisfies $%
(p_{0})$ and $(HG)$. Then for any ball $B(x,R)$ and for any $0<r\leq R/2$,
we have%
\begin{equation}
\sup_{y\notin B(x,r)}g^{B(x,R)}(x,y)\simeq \rho (B(x,r),B(x,R))\simeq
\inf_{y\in B(x,r)}g^{B(x,R)}(x,y).  \label{G-cap}
\end{equation}
\end{proposition}

\begin{proposition}
\label{PG=resk}Assume that the graph $\left( \Gamma ,\mu \right) $ satisfies
$(p_{0})$ and $(HG)$. Fix any ball $B(x,r)$ and denote $B_{k}=B(x,2^{k}r)$
for $k=0,1,...$. Then for all integers $n>m\geq 0$,
\begin{equation}
\sup_{y\notin B_{m}}g^{B_{n}}(x,y)\simeq \sum_{k=m}^{n-1}\rho
(B_{k},B_{k+1})\simeq \inf_{y\in B_{m}}g^{B_{n}}(x,y).  \label{G-capk}
\end{equation}
\end{proposition}

\begin{proposition}
\label{pHG}Assume that $(\Gamma ,\mu )$ satisfies $(p_{0})$. Then
\begin{equation*}
(H)\Longrightarrow (HG)
\end{equation*}
\end{proposition}

For the proof see \cite{B} which improves Proposition 4.3 of \cite{GT2}.

\begin{lemma}
\label{lgseq}If $\left( \Gamma ,\mu \right) $ satisfies $\left( p_{0}\right)
$ then $\left( HG\right) ,\left( \ref{HGrR}\right) ,\left( \ref{HG2R}\right)
$ and $\left( \ref{G-cap}\right) $ are equivalent.
\end{lemma}

\begin{proof}
It is immediate that $\left( HG\right) \Longrightarrow \left( \ref{HGrR}%
\right) \Longrightarrow \left( \ref{HG2R}\right) $. From Proposition \ref%
{PG=res} we have that $\left( HG\right) \Longrightarrow \left( \ref{G-cap}%
\right) $ and $\left( \ref{G-cap}\right) \Longrightarrow \left( \ref{HGrR}%
\right) $ is clear.\ \ The careful reading of \cite[Lemma 10.2]{GT1}
establishes that $\left( \ref{HG2R}\right) \Longrightarrow \left( H\right) $
\ and $\left( HG\right) $ follows by Proposition \ref{pHG}.
\end{proof}

The main result if this section is the following.

\begin{theorem}
\label{tannhar} Assume that $\left( \Gamma ,\mu \right) $ satisfies $\left(
p_{0}\right) ,(H)$. Then
\begin{equation}
\rho (x,R,4R)\leq C_{2}\rho (x,2R,4R),  \label{rrD}
\end{equation}%
and if in addition $\left( BC\right) $ holds then%
\begin{equation}
\rho (x,R,4R)\leq C_{1}\rho (x,R,2R)  \label{rrd1}
\end{equation}%
where $C_{i}>1$ are independent of $x\in \Gamma $ and $R\geq 0.$
\end{theorem}

\begin{proof}
Assume that $\left( \Gamma ,\mu \right) $ satisfies $\left( p_{0}\right)
,(H) $. If $R\leq 16$ the statements follows from $\left( p_{0}\right) ,$ so
we assume that $R>16.$ The first statement is direct consequence of $\left(
H\right) $ and $\left( \ref{G-cap}\right) $. Since $\Gamma $ is connected
there is a path from $x$ to $B^{c}\left( x,4R\right) .$ \ This path has
intersection with $S\left( x,R+1\right) $ in $y_{0}$ and with $S\left(
x,2R-1\right) $ in $z_{0}.$ Along this path we can form a finite
intersecting chain of balls $B\left( x_{i},R/4\right) $ \ with centers on
the path starting with $x_{0}=y_{0}$ and ending with $x_{K}=z_{0}.$ \ It is
clear that $x\notin B\left( x_{i},R/2\right) \subset B\left( x,4R\right) =:B$
hence $g^{B\left( x,4R\right) }\left( x,.\right) $ is harmonic in them and
Harnack inequality and the standard chaining argument can be used to obtain%
\begin{equation*}
g^{B}\left( x,y_{0}\right) \leq Cg^{B}\left( x,z_{0}\right) .
\end{equation*}%
Now using $\left( \ref{G-cap}\right) $ twice it follows that%
\begin{eqnarray*}
\rho \left( x,R,4R\right) &\leq &C\inf_{y\in B\left( x,R\right) }g^{B}\left(
x,y\right) \\
&\leq &Cg^{B}\left( x,y_{0}\right) \leq Cg^{B}\left( x,z_{0}\right) \\
&\leq &C\sup_{z\notin B\left( x,2R\right) }g^{B}\left( x,z\right) \\
&\leq &C\rho \left( x,2R,4R\right) .
\end{eqnarray*}%
Let us prove $\left( \ref{rrd1}\right) $. Let $U=B\left( x,5R\right) ,$ $%
A=B\left( x,R\right) ,$ $D=B\left( x,4R\right) \backslash B\left( x,\frac{3}{%
2}R\right) $. Consider the connected components of $D,$ denote them by $%
D_{i} $ and $S_{i}\left( x,r\right) =S\left( x,r\right) \cap D_{i}$. From
the bounded covering condition it follows that the number of these
components is bounded by $K.$\ Let $\Gamma _{i}=D_{i}\cup \left[ B\left(
x,5R\right) \backslash B\left( x,4R\right) \right] \cup $ $B\left( x,\frac{3%
}{2}R\right) $. \ It is clear that%
\begin{eqnarray*}
\frac{1}{\rho (x,R,2R)} &\leq &\sum_{i=1}^{K}\frac{1}{\rho (B\left(
x,R\right) ,S_{i}\left( x,2R\right) )} \\
&\leq &\frac{K}{\min_{i}\rho (B\left( x,R\right) ,S_{i}\left( x,2R\right) )}.
\end{eqnarray*}%
Let us simply assume that the minimum is obtained for $i=1,$ so that%
\begin{equation}
\frac{\rho (x,R,4R)}{\rho (x,R,2R)}\leq K\frac{\rho (B\left( x,R\right)
,S_{1}\left( x,4R\right) )}{\rho (B\left( x,R\right) ,S_{1}\left(
x,2R\right) )}.  \label{komp1}
\end{equation}%
Let us consider the capacity potential $u\left( y\right) $ between $\Gamma
\backslash B\left( x,5R\right) $ and $B\left( x,R\right) $ which is set zero
on $B\left( x,R\right) $ and $u\left( w\right) =\rho (x,R,5R)$ for $w\in $\ $%
\Gamma \backslash B\left( x,5R\right) $. It is clear that $u\left( y\right) $
is harmonic in $D$. \ Our strategy then is the following. We will compare
potential values of $u$ using the Harnack inequality along a chain of balls
consisting again a bounded number of balls. \ Thanks to the bounded covering
property $B\left( x,5R\right) $ can be covered by a bounded number of balls
of radius $r=R/16.$ We consider the subset of such balls which intersect
with $D_{1}.$ If $B_{i}=B\left( o_{i},r\right) $ is such a ball, it is clear
that $B\left( o_{i},4r\right) $ does not intersect $B\left( x,R\right) $ and
$\Gamma \backslash B\left( x,5R\right) $ and hence $u$ is harmonic in $%
B\left( o_{i},4r\right) $ \ First let $y,y^{\prime }\in D_{1}$ and
\begin{equation*}
\pi =\pi \left( y,y^{\prime }\right) =\left\{ y_{0},y_{1},...y_{N}=y^{\prime
}\right\}
\end{equation*}%
the shortest path connecting them. \

Let us consider a minimal covering of the path by balls. \ Let us pick up
the ball $B\left( o_{i},r\right) $ of smallest index which\ contains $y$
then fix the last point along the path from $y$ to $y_{1}$ $\in $ $\pi
\left( y,y^{\prime }\right) $ which is in this ball and the next one, $z_{1}$
which is not. Now let us pick up the a ball with the smallest index $B\left(
o_{j},r\right) $ which covers $z_{1}$. \ From the triangular inequality%
\begin{equation*}
d\left( o_{i},z_{1}\right) \leq d\left( z_{1},y_{1}\right) +1=r+1
\end{equation*}%
it follows that $z_{1}\in B\left( o_{i},2r\right) $ if $r\geq 1,$ which
means that the elliptic Harnack inequality applied in $B\left(
o_{i},4r\right) $ and then in $B\left( o_{j},4r\right) $ implies that
\begin{equation*}
u\left( y\right) \leq Cu\left( z_{1}\right) .
\end{equation*}%
The procedure can be continued until either $y^{\prime }$ is covered or all
balls are used. Since at least one new point is covered in each step and
only unused balls are selected, the procedure has no loop in it. Since we
started with $K$ balls which cover $B\left( x,5R\right) $ this procedure
does not end before $y^{\prime }$ is covered. \ When $y^{\prime }$ is
covered of course we are ready since at most $K$ balls are used and $K+1$
iterations should be made. This means that
\begin{equation}
u\left( y\right) \leq C^{K+1}u\left( y^{\prime }\right) .  \label{ucomp}
\end{equation}%
\ Let us apply this comparison for $y,y\prime \in S_{1}\left( x,2R\right) $
then for $z,z\prime \in S_{1}\left( x,4R\right) $. \ From the maximum
principle it follows that%
\begin{equation*}
\min_{y\in S\left( x,2R\right) }u\left( y\right) \leq \rho \left( B\left(
x,R\right) ,S_{1}\left( x,2R\right) \right) \leq \max_{y\in S\left(
x,2R\right) }u\left( y\right)
\end{equation*}%
which together with $\left( \ref{ucomp}\right) $ results that%
\begin{equation*}
\rho \left( B\left( x,R\right) ,S_{1}\left( x,2R\right) \right) \simeq
u\left( y\right)
\end{equation*}%
for all $y\in S_{1}\left( x,2R\right) $. The same argument yields that\
\begin{equation*}
\rho \left( B\left( x,R\right) ,S_{1}\left( x,4R\right) \right) \simeq
u\left( z\right)
\end{equation*}%
for all $z\in S_{1}\left( x,4R\right) $. \ Finally let us consider a ray
from $x$ to a $z_{0}\in S_{1}\left( x,4R\right) $ and its intersection $%
y_{0} $ with $S_{1}\left( x,2R\right) $. This ray gives the shortest path
between $y_{0}$ and $z_{0}$ and an other chaining gives that%
\begin{equation*}
\rho \left( B\left( x,R\right) ,S_{1}\left( x,4R\right) \right) \simeq
u\left( z_{0}\right) \simeq u\left( y_{0}\right) \simeq \rho \left( B\left(
x,R\right) ,S_{1}\left( x,2R\right) \right)
\end{equation*}%
which by $\left( \ref{komp1}\right) $ gives the statement.
\end{proof}

\begin{remark}
In the rest of this section the constants $C_{1},C_{2}$ refer to the fixed
constants of $\left( \ref{rrD}\right) $ and $\left( \ref{rrd1}\right) $.
\end{remark}

\begin{corollary}
\label{rnoM} \label{crrd}\label{crrvD}Under the conditions of Theorem \ref%
{tannhar} the inequalities
\begin{equation}
\rho (x,2R,4R)\leq \left( C_{1}-1\right) \rho (x,R,2R),  \label{rd1v1}
\end{equation}%
\begin{equation}
\rho (x,R,2R)\leq \left( C_{2}-1\right) \rho (x,2R,4R)  \label{rd2v1}
\end{equation}%
and%
\begin{equation}
\frac{1}{C_{1}-1}\rho (x,2R,4R)\leq \rho (x,R,2R)\leq \left( C_{2}-1\right)
\rho (x,2R,4R)  \label{rd3v1}
\end{equation}%
hold.
\end{corollary}

\begin{proof}
The first two statements follow in the same way from the easy observation
that%
\begin{equation*}
\rho (x,R,4R)\geq \rho (x,R,2R)+\rho (x,2R,4R).
\end{equation*}
\end{proof}

We end this section with some observations which might be interesting on
their own. Some consequences of $\left( \ref{rd1v1}\right) $ and $\left( \ref%
{rd2v1}\right) $ ( in fact consequences of the elliptic Harnack inequality)
are derived and bounds on the volume growth and the mean exit time are
deduced.

\begin{remark}
Let us observe that from $\left( \ref{rd1v1}\right) $\ it follows that the
graph is transient if $C_{1}<2$ and in this case $\left( \ref{rd2v1}\right) $
ensures that the decay is not faster than polynomial. \ Similarly from $%
\left( \ref{rd2v1}\right) $ it follows that the graph is recurrent if $%
C_{2}\leq 2$ and $\left( \ref{rd1v1}\right) $ ensures that the resistance
not increases faster then polynomial. It is clear from $\left( \ref{rd3v1}%
\right) $ that
\begin{equation*}
\left( C_{1}-1\right) \left( C_{2}-1\right) \geq 1.
\end{equation*}
\end{remark}

\begin{remark}
From the $\left( p_{0}\right) $,$\left( \ref{rd1v1}\right) $ and $\left( \ref%
{rd2v1}\right) $ one can obtain easily the following inequalities.%
\begin{equation}
\rho \left( x,R,2R\right) \leq C\rho \left( x,1,2\right) R^{\log _{2}\left(
C_{1}-1\right) }\leq \frac{C}{\mu \left( x\right) }R^{\log _{2}\left(
C_{1}-1\right) },  \label{c1}
\end{equation}%
\begin{equation}
\rho \left( x,R,2R\right) \geq C\rho \left( x,1,2\right) R^{-\log _{2}\left(
C_{2}-1\right) }\geq \frac{c}{\mu \left( x\right) }R^{-\log _{2}\left(
C_{2}-1\right) }.  \label{c2}
\end{equation}
\end{remark}

\begin{remark}
Barlow in \cite{B} proved that $\left( p_{0}\right) $ and the elliptic
Harnack inequality imply%
\begin{equation}
\left\vert V\left( x,R\right) \right\vert \leq CR^{1+\theta }  \label{barl}
\end{equation}%
where $\theta =\log _{3}H$ and $H$ is the constant in the Harnack
inequality. The combination of Corollary \ref{crv>r2} and $\left( \ref{c1}%
\right) $ results a lower bound for the volume growth: \
\begin{equation*}
R^{2}\leq \left( V\left( x,2R\right) -V\left( x,R\right) \right) \rho \left(
x,R,2R\right) \leq V\left( x,2R\right) \frac{C}{\mu \left( x\right) }R^{\log
_{2}\left( C_{1}-1\right) },
\end{equation*}%
and so $\left( p_{0}\right) +\left( H\right) +\left( BC\right) $ implies a
lower bound for $V$:
\begin{equation}
V\left( x,2R\right) -V\left( x,R\right) \geq c\mu \left( x\right) R^{2-\log
_{2}\left( C_{1}-1\right) }.  \label{v2}
\end{equation}%
In particularly if $V\left( x,R\right) \leq CR^{\alpha }$ then
\begin{equation}
C_{1}\geq 2^{2-\alpha }+1.  \label{c1aa}
\end{equation}%
Similarly by $\left( H\right) $ and $\left( \ref{c2}\right) $%
\begin{equation}
E\left( x,2R\right) \geq c\rho \left( x,R,2R\right) V\left( x,R\right) \geq
cV\left( x,R\right) R^{-\log _{2}\left( C_{2}-1\right) }  \label{v1}
\end{equation}%
which means that $\left( p_{0}\right) +\left( H\right) $ implies an upper
bound for $V\,$:%
\begin{equation*}
V\left( x,R\right) \leq CE\left( x,R\right) \mu \left( x\right) R^{\log
_{2}\left( C_{2}-1\right) }
\end{equation*}%
Similarly to $\left( \ref{c1aa}\right) $ we get
\begin{equation}
C_{2}\geq 2^{\alpha _{1}-\beta }+1.  \label{c2ab}
\end{equation}%
if we assume that \ $E\left( x,R\right) \leq CR^{\beta }$ and $V\left(
x,R\right) \geq c\mu \left( x\right) R^{\alpha _{1}}.$
\end{remark}

\begin{remark}
\label{re>r2}We can restate the above observations starting from $\left( \ref%
{rv>dis^2}\right) $ and using $\left( \ref{rd2v1}\right) $ and $\left(
H\right) .$%
\begin{eqnarray*}
R^{2} &\leq &\rho \left( x,R,2R\right) \left( V\left( x,2R\right) -V\left(
x,R\right) \right) \\
&\leq &\left( C_{2}-1\right) \rho \left( x,2R,4R\right) V\left( x,2R\right)
\\
&\leq &\left( C_{2}-1\right) E\left( x,4R\right) .
\end{eqnarray*}%
This means that $\left( p_{0}\right) +\left( H\right) $ implies
\begin{equation*}
E\left( x,R\right) \geq cR^{2}.
\end{equation*}
\end{remark}

The next corollary highlights the connection between the volume growth and
resistance properties implied by the elliptic Harnack inequality. \
Particularly an upper bound for the volume of a ball (similar to one given
in \cite{B} ) is provided and complemented with a lower bound. \

\begin{corollary}
\label{csumm}Assume that $\left( \Gamma ,\mu \right) $ satisfies $\left(
p_{0}\right) $ and $\left( H\right) .$ Then there are constants $%
C,c>0,C_{1},C_{2}>1$ and $\gamma _{2}=\log _{2}\left( C_{2}-1\right) $ such
that for all $x\in \Gamma ,R>0$%
\begin{equation*}
V\left( x,R\right) \leq CE\left( x,R\right) \mu \left( x\right) R^{\gamma
_{2}}
\end{equation*}%
and%
\begin{equation*}
E\left( x,R\right) \geq cR^{2}.
\end{equation*}%
In addition if $\left( BC\right) $ is satisfied then there is a $\gamma
_{1}=\log _{2}\left( C_{1}-1\right) $ such that%
\begin{equation*}
V\left( x,2R\right) \geq c\mu \left( x\right) R^{2-\gamma _{1}}.
\end{equation*}
\end{corollary}

\section{Harnack graphs}

\setcounter{equation}{0}\label{sER1}The notion of Harnack graphs was coined
by Barlow (personal communication) some time ago in order to have a concise
name for graphs which satisfy the elliptic Harnack inequality. \ At that
time the investigations\ were focused on fractals and fractal like graphs in
which the space-time scaling function was $R^{\beta }.$ In this section we
focus on graphs which on one hand satisfy the elliptic Harnack inequality on
the other hand satisfy the triplet $\left( p_{0}\right) ,\left( VD\right)
,\left( TC\right) $ already used in Section \ref{sbasic}. The results of
\cite{Tweak} show that this set of conditions is strong enough to obtain
heat kernel estimates. The main result of this section is the following
theorem.

\begin{theorem}
\label{tERE}If for a weighted graph $\left( \Gamma ,\mu \right) $ the
conditions $\left( p_{0}\right) ,\left( VD\right) ,\left( H\right) $ and $%
\left( wTC\right) $ hold then%
\begin{equation*}
E\left( x,2R\right) \simeq \rho \left( x,R,2R\right) v\left( x,R,2R\right) .
\end{equation*}
\end{theorem}

\begin{proposition}
\label{te<rv}If $\left( p_{0}\right) ,\left( wTC\right) $ holds for $\left(
\Gamma ,\mu \right) $ then there is a $C>1$ such that for all $x\in \Gamma
,R>0$%
\begin{equation*}
E\left( x,2R\right) \leq C\rho \left( x,R,5R\right) v\left( x,R,5R\right) .
\end{equation*}
\end{proposition}

\begin{proof}
Let us consider the annulus $D=B\left( x,5R\right) \backslash B\left(
x,R\right) $. We apply Corollary \ref{cE<rm} for $A=B\left( x,R\right)
,B=B\left( x,5R\right) $ to obtain%
\begin{equation*}
\rho \left( x,R,5R\right) v\left( x,R,5R\right) \geq E_{a}\left(
T_{B}\right) .
\end{equation*}%
Now we use the Markov property for the stopping time $T_{B\left( x,3R\right)
}.$ It is clear that the walk started in $B\left( x,R\right) $ and leaving $%
B\left( x,5R\right) $ should cross $S\left( x,3R\right) .$ \ Denote the
first hitting (random) point by $\xi .$ It is also evident that the walk
continued from $\xi $ should leave $B\left( \xi ,2R\right) $ before it
leaves $B\left( x,5R\right) $.\ This means that
\begin{eqnarray*}
&&\rho \left( x,R,5R\right) v\left( x,R,5R\right) \\
&\geq &E_{a}\left( T_{B}\right) \geq \min_{y\in S\left( x,3R\right) }E\left(
y,2R\right) \\
&\geq &cE\left( x,2R\right) ,
\end{eqnarray*}%
where the last inequality follows from the repeated use of $\left(
wTC\right) .$
\end{proof}

\begin{proof}[Proof of Theorem \protect\ref{tERE}]
The lower estimate is easy. Denote $B=B\left( x,2R\right) .$ We know that $%
\left( H\right) $ implies $\left( HG\right) ,\left( \ref{G-cap}\right) $ and
consequently%
\begin{eqnarray}
E\left( x,2R\right) &=&\sum_{y\in B}G^{B}\left( x,y\right)  \label{erle} \\
&\geq &\sum_{y\in B\left( x,R\right) }g^{B}\left( x,y\right) \mu \left(
y\right)  \notag \\
&\geq &c\rho \left( x,R,2R\right) V\left( x,R\right)  \notag \\
&\geq &c\rho \left( x,R,2R\right) v\left( x,R,2R\right) .  \notag
\end{eqnarray}%
The upper estimate uses the fact that the Harnack inequality implies the
doubling property of the resistance. From Proposition \ref{te<rv} we have
that%
\begin{equation*}
E\left( x,2R\right) \leq C\rho \left( x,R,5R\right) v\left( x,R,5R\right)
\end{equation*}%
but $\left( \ref{rrd1}\right) $ and \ volume doubling gives that%
\begin{eqnarray*}
E\left( x,2R\right) &\leq &C\rho \left( x,R,8R\right) v\left( x,R,5R\right)
\\
&\leq &C\rho \left( x,R,2R\right) V\left( x,5R\right) \\
&\leq &C\rho \left( x,R,2R\right) v\left( x,R,2R\right) .
\end{eqnarray*}
\end{proof}

The next lemma is weaker than the observation in Remark \ref{re>r2} but the
proof is so easy that we include it here.

\begin{lemma}
\label{lE>r2}If for a weighted graph $\left( \Gamma ,\mu \right) $ the
conditions $\left( p_{0}\right) ,\left( VD\right) ,\left( H\right) $ hold
then%
\begin{equation}
E\left( x,R\right) \geq cR^{2}.  \label{E>R2}
\end{equation}
\end{lemma}

\begin{proof}
As we have seen in $\left( \ref{erle}\right) $
\begin{equation*}
E\left( x,R\right) \geq c\rho \left( x,R/2,R\right) v\left( x,R/2,R\right)
\end{equation*}%
follows from the conditions and from $\left( \ref{rv>dis^2}\right) $ we
obtain the statement.
\end{proof}

\section{Resistance condition on Harnack graphs}

\setcounter{equation}{0}$\label{sER2}$In order to receive a set of
conditions which is based on volume and resistance properties we collect the
properties of the product of the functions $\rho \left( x,R,2R\right) $ \
and $v\left( x,R,2R\right) .$ Under the new set of conditions the Einstein
relation holds again. This case has an interesting point. \ The proof relies
on \ that the product $\rho v$ satisfies the anti-doubling property. \ In
this section we show that there are several conditions equivalent to the
anti-doubling property of $\rho v$. At the end of the section a concise
condition on the local Green kernel is presented (cf. $\left( \ref{gub}%
\right) $,$\left( \ref{glb}\right) $ and $\left( g\right) $ ) which is
equivalent to $\left( ER\right) +\left( H\right) $ provided the graphs
satisfies $\left( p_{0}\right) $ and $\left( VD\right) .$ The two-sided
bound on the local Green kernel used in $\left( \cite{GT1}\right) $ and $%
\left( \cite{BB}\right) $. The new relation $\left( g\right) $ is joint
generalization of them and leads to characterization of graphs having heat
kernel estimates of local type and parabolic Harnack inequalities ( cf. \cite%
{Tweak} ).

Let us start with an interesting observation. \ The anti-doubling property
of $\rho v$ follows from a stronger assumption, from the assumption $\left(
\mathbf{\rho v}\right) $: $\rho \left( x,R,2R\right) v\left( x,R,2R\right) $
is basically independent of the reference point $x:$ that is there is a $C>0$
such that%
\begin{equation}
\rho \left( x,R,2R\right) v\left( x,R,2R\right) \simeq \rho \left(
y,R,2R\right) v\left( y,R,2R\right) .  \label{rv}
\end{equation}

\begin{proposition}
\label{padres}Assume that for $\left( \Gamma ,\mu \right) $ $\left(
p_{0}\right) ,\left( VD\right) ,\left( H\right) $ and $\left( \rho v\right) $
hold. Then there is an $A=A_{\rho v}>1$ such that anti-doubling for $\rho v$
holds:%
\begin{equation}
\rho \left( x,AR,2AR\right) v\left( x,AR,2AR\right) \geq 2\rho \left(
x,R,2R\right) v\left( x,R,2R\right)  \label{adrv}
\end{equation}%
for all $x\in \Gamma .$
\end{proposition}

\begin{proof}
Assume that $R>R_{0},$ otherwise the statement follows from $\left(
p_{0}\right) .$ Let $A=B\left( x,R\right) ,B=B\left( x,2R\right)
,D=B\backslash A$ where $R=4kr$ for a $r\geq 1$.\ Denote by $\xi _{i}$ the
location of the first hit of $\partial B(x,\left( 2\left( k+i\right) \right)
2r)$ for $i=0...k-1$. First by Corollary \ref{cE<rm}%
\begin{equation*}
w_{x}\left( R\right) :=\rho \left( x,R,2R\right) v(x,R,2R)\geq E_{a}\left(
T_{B}\right)
\end{equation*}%
It is evident that the exit time $T_{B}$ in $\Gamma ^{a}$ satisfies%
\begin{equation*}
T_{B}\geq \sum_{i=0}^{k-1}T_{B\left( \xi _{i},2r\right) }
\end{equation*}%
and consequently by $\left( \ref{ta}\right) $%
\begin{equation*}
E_{a}\left( T_{B}\right) \geq \sum_{i=0}^{k-1}E\left( \xi _{i},2r\right) .
\end{equation*}%
The terms on the r.h.s can be estimated using the $\left( H\right) $ as in $%
\left( \ref{erle}\right) $ to obtain%
\begin{eqnarray*}
E_{a}\left( T_{B}\right) &\geq &\sum_{i=0}^{k-1}E\left( \xi _{i},2r\right) \\
&\geq &\sum_{i=0}^{k-1}c\min_{z\in B}w_{z}\left( r\right) \geq ckw_{x}\left(
r\right) ,
\end{eqnarray*}%
where $\left( \rho v\right) $ was used in the last step. \ Finally choosing $%
A_{\rho v}=k$ so that $k\geq 2/c$ we get the statement.
\end{proof}

\begin{definition}
We say that the condition $\left( \mathbf{E}\right) $ holds on $\Gamma $ if
\begin{equation}
E\left( x,R\right) \simeq E\left( y,R\right) .  \label{e}
\end{equation}
\end{definition}

\begin{remark}
Let us observe that under the condition of Proposition \ref{padres} with
some increase of the number of iterations it follows that for the function
\begin{equation*}
F\left( R\right) =\inf_{x\in \Gamma }\rho \left( x,R,2R\right) v\left(
x,R,2R\right) ,
\end{equation*}%
the anti-doubling property%
\begin{equation*}
F\left( A_{F}R\right) \geq 2F\left( R\right)
\end{equation*}%
holds. Of \ course the same applies in the presence of $\left( E\right) $ as
a consequence of Proposition \ref{padE}.
\end{remark}

\begin{theorem}
\label{tERrrv} If for a weighted graph $\left( \Gamma ,\mu \right) $ the
conditions $\left( p_{0}\right) ,\left( VD\right) ,\left( H\right) $ and $%
\left( \ref{adrv}\right) $ hold then%
\begin{equation*}
E\left( x,2R\right) \simeq \rho \left( x,R,2R\right) v\left( x,R,2R\right)
\end{equation*}
\end{theorem}

\begin{proof}
We assume that $R>R_{0},$ otherwise the statement is consequence of $\left(
p_{0}\right) $. The lower estimate can be deduced as in $\left( \ref{erle}%
\right) $. \ The upper estimate uses Proposition \ref{PG=resk}. Denote $%
M\geq A_{\rho v},L=M^{2}$ fixed constants ($A_{\rho v}$ is from Proposition %
\ref{PG=resk} ), $\ R_{k}=M^{k}$, $B_{k}=B(x,R_{k})$ and let $n$ be the
minimal integer so that $LR<R_{n}$. We have
\begin{eqnarray}
E(x,2R) &\leq &E(x,R_{n})=\sum_{y\in B_{n}}g^{B_{n}}(x,y)\mu (y)
\label{sum1} \\
&=&\sum_{y\in B_{0}}g^{B_{n}}(x,y)\mu (y)+\sum_{m=0}^{n-1}\sum_{y\in
B_{m+1}\setminus B_{m}}g^{B_{n}}(x,y)\mu (y).  \label{sum2}
\end{eqnarray}%
It follows from $\left( p_{0}\right) $ that the first term on the right hand
side of (\ref{sum2}) -- the sum over $B_{0}$ -- is majorized by a multiple
of a similar sum over $B_{1}\setminus B_{0}$, which is a part of the second
term. Estimating $g^{B_{n}}$ by $\left( \ref{G-capk}\right) $ we obtain
\begin{align*}
& E\left( x,2R\right) \\
& \leq E(x,LR) \\
& \leq C\sum_{m=0}^{n}\left[ \sum_{k=m}^{n}\rho (x,R_{k},R_{k+1})\right]
v\left( x,R_{m},R_{m+1}\right) \\
& \leq C\sum_{k=0}^{n}\left[ \sum_{m=0}^{k}v\left( x,R_{m},R_{m+1}\right) %
\right] \rho \left( x,R_{k},R_{k+1}\right) \\
& \leq C\sum_{k=0}^{n}\rho \left( x,R_{k},R_{k+1}\right) V(x,R_{k+1}) \\
& \leq C\sum_{k=0}^{n}\rho \left( x,R_{k},R_{k+1}\right) v(x,R_{k},R_{k+1}).
\end{align*}%
Now we use the anti-doubling property of $\rho v$ which yields
\begin{eqnarray*}
&\leq &C\rho (R_{n-1},R_{n})v(R_{n-1},R_{n})\sum_{k=0}^{n}2^{k-n} \\
&\leq &C\rho v(R_{n-2},R_{n-1})\leq C\rho v\left( x,R,LR\right) .
\end{eqnarray*}
\end{proof}

\begin{corollary}
If for a weighted graph $\left( \Gamma ,\mu \right) $ the conditions $\left(
p_{0}\right) ,\left( VD\right) ,\left( H\right) $ and $\left( \rho v\right) $
hold then%
\begin{equation*}
E\left( x,2R\right) \simeq \rho \left( x,R,2R\right) v\left( x,R,2R\right)
\end{equation*}
\end{corollary}

\begin{proof}
The statement is direct consequence of Proposition \ref{padres} and Theorem %
\ref{tERrrv}.
\end{proof}

\begin{remark}
One can check that under $\left( p_{0}\right) ,\left( VD\right) $ \ and $%
\left( H\right) $%
\begin{equation}
\left( wTC\right) \Longleftrightarrow \left( ER\right) \Longleftrightarrow
\left( TC\right) \Longleftrightarrow \left( \ref{adrv}\right) .  \label{star}
\end{equation}%
\newline
The main line of the proof is indicated in the following diagrams assuming $%
\left( p_{0}\right) $:%
\begin{equation*}
\left( \ref{adrv}\right) +\left( VD\right) +\left( H\right) \Longrightarrow
\left( ER\right) ,\left( TD\right) ,\left( TC\right) ,\left( \overline{E}%
\right) ,\left( wTC\right) ,
\end{equation*}%
follows from Proposition $\left( \ref{tERrrv}\right) $ and%
\begin{equation*}
\left( wTC\right) +\left( VD\right) +\left( H\right) \Longrightarrow \left(
ER\right) ,\left( TD\right) ,\left( TC\right) ,\left( \overline{E}\right)
,\left( \ref{adrv}\right) .
\end{equation*}%
from Theorem $\left( \ref{tERE}\right) $.
\end{remark}

\begin{definition}
We introduce upper and lower bound for the Green kernel. There are $C,c>0$
such that for all $x\in \Gamma ,R>0,B=B\left( x,2R\right) ,A=B\left(
x,R\right) $
\begin{equation}
\max_{y\notin A}g^{B}\left( x,y\right) \leq C\frac{E\left( x,2R\right) }{%
V\left( x,R\right) }  \label{gub}
\end{equation}%
\begin{equation}
\min_{y\in A}g^{B}\left( x,y\right) \geq c\frac{E\left( x,2R\right) }{%
V\left( x,R\right) }  \label{glb}
\end{equation}%
If both satisfied it will be referred by $\left( g\right) .$
\end{definition}

\begin{theorem}
Assume that for a weighted graph $\left( \Gamma ,\mu \right) $ the
conditions $\left( p_{0}\right) ,\left( VD\right) $ hold. Then%
\begin{equation*}
\left( g\right) \Leftrightarrow \left( H\right) +\left( ER\right)
\end{equation*}
\end{theorem}

\begin{proof}
It is clear that $\left( \ref{gub}\right) +\left( \ref{glb}\right)
\Longrightarrow \left( \ref{HG2R}\right) $ which is equivalent to $\left(
HG\right) $ by Lemma \ref{lgseq}. We know by Proposition \ref{PG=>H} $\ $%
that $\left( HG\right) \Longrightarrow \left( H\right) $ and by Proposition %
\ref{PG=res} that $\left( HG\right) $ implies $\left( \ref{G-cap}\right) .$
We can use $\left( \ref{G-cap}\right) +\left( \ref{glb}\right) $ to obtain
\begin{equation*}
E\left( x,2R\right) \leq C\rho \left( x,R,2R\right) v\left( x,R,2R\right)
\end{equation*}%
while the lower estimate follows from $\left( \ref{G-cap}\right) +\left( \ref%
{gub}\right) ,$ so we have $\left( ER\right) $. The reverse implication
follows from the fact that $\left( H\right) \Longrightarrow \left( HG\right)
\Longrightarrow \left( \ref{G-cap}\right) $ which can be combined with $%
\left( ER\right) $ to receive $\left( g\right) .$
\end{proof}

\begin{corollary}
Assume that for a weighted graph $\left( \Gamma ,\mu \right) $ the
conditions $\left( p_{0}\right) ,\left( VD\right) $ hold. Then%
\begin{equation*}
\left( g\right) \Longleftrightarrow \left( H\right) +\left( \ast \right)
\end{equation*}%
where $\left( \ast \right) $ can be any of the conditions in $\left( \ref%
{star}\right) $, in particular $\left( \ast \right) $ can be $\left( \ref%
{adrv}\right) $, the anti-doubling property of $\rho v$, or $\left(
wTC\right) .$
\end{corollary}

\begin{remark}
One should note that $\left( \ref{gub}\right) $ follows easily from the
elliptic mean value inequality ( cf. \cite{Tfull}, ) but $\left( \ref{glb}%
\right) $ is stronger than a reversed kind of anti-mean value inequality.
\end{remark}

\section{The strong anti-doubling property}

\setcounter{equation}{0}\label{ssad} The anti-doubling property has a
stronger form (see below $\left( \ref{sadF}\right) $), which is essential
for off-diagonal heat kernel lower bounds). \ It is equivalent to%
\begin{equation}
\frac{\rho \left( x,R,2R\right) v\left( x,R,2R\right) }{\rho \left(
x,r,2R\right) v\left( x,r,2R\right) }\geq c\left( \frac{R}{r}\right) ^{\beta
_{1}}  \label{adrrvbeta1}
\end{equation}%
for some $c>0,\beta _{1}>1$ and for all $x\in \Gamma ,R>r>0$. In this
section we deduce $\left( \ref{adrrvbeta1}\right) $ working under the
assumption $\left( p_{0}\right) ,\left( VD\right) $ and $\left( H\right) .$
\ We will see that $\left( \ref{adrrvbeta1}\right) $ or $\left( \ref{sadF}%
\right) $ follows if we assume that the graph is homogeneous with respect to
the function $\rho v$ in $x\in \Gamma .$ This condition seems to be
necessary for the strong anti-doubling property but we can not prove or
disprove the necessity.

\begin{lemma}
If $\left( ER\right) $ holds then the following anti-doubling properties are
equivalent (with different constants).\newline
1. There is an $A>1$ such that%
\begin{equation}
E\left( x,AR\right) \geq 2E\left( x,R\right)  \label{ad1}
\end{equation}%
for all $x,R.$2. There is an $A^{\prime }>1$ such that%
\begin{equation}
\rho (x,A^{\prime }R,2A^{\prime }R)v(x,A^{\prime }R,2A^{\prime }R)\geq 2\rho
(x,R,2R)v(x,R,2R)  \label{ad2}
\end{equation}%
for all $x,R.$
\end{lemma}

\begin{proof}
Let us apply $\left( ER\right) $ and $\left( \ref{ad1}\right) \ $%
iteratively. Set $A^{\prime }=A^{k}$ for some $k>1$%
\begin{eqnarray*}
&&\rho (x,A^{\prime }R,2A^{\prime }R)v(x,A^{\prime }R,2A^{\prime }R) \\
&\geq &cE\left( x,A^{\prime }R\right) \\
&\geq &c2^{k}E\left( x,R\right) \\
&\geq &c2^{k}c^{\prime }\rho (x,R,2R)v(x,R,2R).
\end{eqnarray*}%
So if $k=\left\lceil -\log \left( cc^{\prime }\right) \right\rceil ,$ $%
A^{\prime }=2^{k}$ we receive $\left( \ref{ad2}\right) $. \ The reverse
implication works in the same way.
\end{proof}

For the strong anti-doubling property of
\begin{equation}
F\left( R\right) =\inf_{x\in \Gamma }E\left( x,R\right)  \label{Fdef}
\end{equation}%
first we show that it is at least linear. We also note that $\left( E\right)
$ implies $\left( wTC\right) .$

The combination of $\left( E\right) $ and $\left( \ref{Fdef}\right) $
clearly gives that%
\begin{equation*}
E\left( x,R\right) \simeq F\left( R\right) .
\end{equation*}

\begin{lemma}
\label{lfl}If
\begin{equation*}
E\left( x,R\right) \simeq F\left( R\right)
\end{equation*}%
then for all $L\in
%TCIMACRO{\U{2115} }%
%BeginExpansion
\mathbb{N}
%EndExpansion
,R>1$
\begin{equation*}
F\left( LR\right) \geq LF\left( R\right) .
\end{equation*}
\end{lemma}

\begin{proof}
Let us fix an $x=x_{\varepsilon ,R}$ for which
\begin{equation*}
F(LR)+\varepsilon \geq E(x,LR)
\end{equation*}%
and use strong Markov property.%
\begin{eqnarray*}
E(x,LR) &\geq &E\left( x,\left( L-1\right) R\right) +\min_{z\in \partial
B\left( x,\left( L-1\right) R\right) }E\left( z,R\right) \\
&\geq &...\geq L\min_{z\in B\left( x,LR\right) }E\left( x,R\right) \geq
LF\left( R\right) .
\end{eqnarray*}%
Since $\varepsilon $ was arbitrary we get the statement,
\end{proof}

\begin{proposition}
\label{pFsr}If $(p_{0}),\left( VD\right) ,\left( E\right) $ and $(H)$ hold
then there are $B_{F}>A_{F}>1$ such that for all $R\geq 1$%
\begin{equation}
F(A_{F}R)\geq B_{F}F(R).  \label{sadF}
\end{equation}
\end{proposition}

\begin{proof}
The proof starts with a special choice of the reference point. We fix an $%
\varepsilon >0$ small constant, which will be chosen later. Assume that $%
R\geq 1$ and assign an $x=x_{\varepsilon ,R}\in \Gamma $ to $\varepsilon $ \
and $R$ satisfying
\begin{equation*}
F(3R)+\varepsilon \geq E(x,3R).
\end{equation*}%
Let us denote by $\tau _{A}$ the first hitting time of a set $A=B\left(
x,R\right) $ and denote $B=B\left( x,3R\right) ,D=B\left( s,2R\right) .$
Also denote by $\xi =X_{T_{D}}\in \partial B(x,2R)$ and split the history of
the walk according to $T_{D}$. Using the strong Markov property $E(x,3R)$
can be estimated from below by
\begin{eqnarray*}
E(x,3R) &\geq &E(x,2R)+\mathbb{E}_{x}\left( \mathbb{E}_{\xi }\left[
T_{B}\wedge \tau _{A}\right] \right) \\
&&+\mathbb{E}(I\left[ T_{B}>\tau _{A}\right] (T_{B}-\tau _{A})) \\
&\geq &F(2R)+\mathbb{E}_{x}\left( \mathbb{E}\left( \xi ,R\right) \right) \\
&&+\mathbb{E}_{x}\left[ I\left( T_{B}>\tau _{A}\right) \mathbb{E}_{\xi
}\left( T_{B}\right) \right] \\
&\geq &2F\left( R\right) +F\left( R\right) \\
&&+\mathbb{E}_{x}\left[ I\left( T_{B}>\tau _{A}\right) \mathbb{E}_{\xi
}\left( T_{B}\right) \right] ,
\end{eqnarray*}%
where in the last step Lemma \ref{lfl} was used. \ The third term contains
the sub-case when the walk reaches $\partial B(x,2R)$ then returns to $A,$
before it leaves. Let us denote this return site by $\zeta =X_{k}:k=\min
\left\{ i:T_{D}<i,X_{i}\in A\right\} $. Using this we get
\begin{align*}
& \mathbb{E}_{x}\left[ I\left( T_{B}>\tau _{A}\right) \mathbb{E}_{\xi
}\left( T_{B}\right) \right] \\
& =\mathbb{E}_{x}\left( \mathbb{P}_{\xi }\left( T_{B}>\tau _{A})))E(\zeta
,2R\right) \right) \\
& \geq \mathbb{E}_{x}\left( P_{\xi }(T_{B}>\tau _{A})F(2R)\right) \geq
\min_{w\in \partial B\left( x,2R\right) }\mathbb{P}_{w}(T_{B}>\tau
_{A})F(2R).
\end{align*}%
The probability in the above expression can be estimated using the elliptic
Harnack inequality (as in Theorem \ref{tannhar}) to get, that
\begin{equation*}
\min_{w\in \partial B\left( x,2R\right) }\mathbb{P}_{w}(T_{B}>\tau _{A})\geq
c\frac{\rho (2R,3R)}{\rho (R,3R)}\geq c=:c_{0}.
\end{equation*}%
Now we have the inequality%
\begin{eqnarray*}
F\left( 3R\right) +\varepsilon &\geq &3F\left( R\right) +c_{0}F\left(
2R\right) \\
&\geq &3F\left( R\right) +c_{0}2F\left( R\right) ,
\end{eqnarray*}%
which means that if $c_{F}>\frac{2\varepsilon }{F(R)},$ i.e. $\varepsilon
\leq \frac{1}{2}F(1)c_{0}$ the statement follows with $\ A_{F}=3$, $B_{F}=3+%
\frac{c_{0}}{2}.$
\end{proof}

\begin{remark}
One can, of course, formulate the strong anti-doubling property for $E\left(
x,R\right) $ or for $\rho v$ with a slight increase of $A,$ but it seems
more natural to state it for $F.$
\end{remark}

\begin{remark}
It is also clear that $F$ inherits from $E$ or $\rho v$ that $F\left(
R\right) \geq cR^{2}.$
\end{remark}

\section{List of lettered conditions}

\begin{equation*}
\begin{array}{lllll}
\text{abbreviation} &  & \text{refers to} &  & \text{name} \\
\left( BC\right) &  & \text{Definition \ref{BC}} &  & \text{bounded covering
condition} \\
\left( VD\right) &  & \left( \ref{PD1V}\right) &  & \text{volume doubling
property} \\
\left( wVC\right) &  & \left( \ref{wVC}\right) &  & \text{weak volume
comparison} \\
\left( TC\right) &  & \left( \ref{TC}\right) &  & \text{time comparison
principle} \\
\left( wTC\right) &  & \left( \ref{pd2e}\right) &  & \text{weak time
comparison} \\
\left( TD\right) &  & \left( \ref{TD}\right) &  & \text{time doubling} \\
\left( ER\right) &  & \left( \ref{ER}\right) &  & \text{Einstein relation}
\\
\left( \rho v\right) ,\left( E\right) &  & \left( \ref{rv}\right) ,\left( %
\ref{e}\right) &  & \Gamma \text{ is homogeneous w.r.t. to }\rho v\text{ or }%
E \\
\left( p_{0}\right) &  & \left( \ref{p0}\right) &  & \text{controlled weights%
} \\
\left( H\right) &  & \left( \ref{H}\right) &  & \text{elliptic Harnack
inequality} \\
\left( \overline{E}\right) &  & \left( \ref{Ebar}\right) &  & \text{%
condition e-bar} \\
\left( HG\right) &  & \left( \ref{HG2}\right) &  & \text{Harnack inequality
for }g \\
\left( g\right) &  & \left( \ref{glb}\right) +\left( \ref{gub}\right) &  &
\text{bounds for }g%
\end{array}%
\end{equation*}

\end{document}